%% file: planner-pm10.11_3_.tex
\documentclass[a4paper,12pt]{article}

\usepackage{geometry}
\usepackage{indentfirst}

\usepackage{needspace}
\usepackage{amsmath,amssymb,amsfonts}
\usepackage{mathrsfs}
\usepackage{bm}
\usepackage{bbm}
\usepackage{amsthm}
\theoremstyle{definition}

\newtheorem{The}{Theorem}
\newtheorem{Lem}[The]{Lemma}
\newtheorem{Cor}[The]{Corollary}
\newtheorem{Pro}[The]{Proposition}

\newtheorem{Cla}{Claim}

\usepackage{booktabs,longtable,multirow,tabularx,threeparttable,float}
\usepackage{graphicx,rotating,subcaption,caption,tikz}
\usepackage{placeins}
\usepackage[noend]{algpseudocode}
\usepackage{algorithm}
\usepackage[numbers,sort&compress]{natbib}
\usepackage{enumerate,multicol,pifont,xcolor}
\numberwithin{equation}{section}
\usepackage[colorlinks,linkcolor=blue,citecolor=blue,urlcolor=blue]{hyperref}
\usepackage[normalem]{ulem}
\graphicspath{{figures/}}

\title{Perfect matchings in planar 3-connected graphs:\protect\\
girth five and minimum degree five}
\author{	Xinyu Dai$^{1}$\quad}
\date{{\small
		$^{1}$School of Mathematics and Statistics, Minnan Normal University,
		Zhangzhou, Fujian 363000, P.R. China\\
		[0.5ex]
		E-mails: daixinyu03520@163.com}
	}

\begin{document}
\maketitle

\begin{abstract}
Let $\Phi(G)$ denote the number of perfect matchings of a graph $G$.
Goedgebeur  et al. posted two question about the perfect matchings of  planar 3-connected graphs:   does there exist an infinite family of matchable planar 3-connected graphs of girth 5 with a constant number of perfect matchings? Whether there exist infinitely many planar 3-connected graphs with minimum
degree 5 which have a constant number of perfect matchings.

We present answers to those questions. First, we construct, for every even $n\ge196$, a simple planar
3-connected graph of order $n$ and girth five with exactly $c$ perfect
matchings, where $c$ is a fixed positive integer.
Second, for a simple planar 3-connected graph $G$ with a perfect
matching and minimum degree at least five, we show that
$\Phi(G)\ge2^{\lceil(|V(G)|+20)/22\rceil}$.
This implies that, for each positive integer $t$, only finitely
many such graphs, up to isomorphism, have at most $t$ perfect matchings.

\end{abstract}
\section{Introduction}

All graphs considered in this paper are finite and simple.
For terminology and notation not defined here, we follow Bondy and
Murty~\cite{Bondy2008} for graph theory and Lov\'asz and
Plummer~\cite{Lovasz1986} for matching theory.
For a graph $G$, let $V(G)$ and $E(G)$ denote its vertex set
and edge set, respectively. The \emph{order} of $G$ is $|V(G)|$.
A \emph{matching} of a graph $G$ is a set of edges no two of
which have a common vertex. A matching $M$ of $G$ is a
\emph{perfect matching} if every vertex of $G$ is incident
with exactly one edge of $M$.
A graph is \emph{cubic} if every vertex has degree three.
The \emph{girth} of a graph containing a cycle is the length
of a shortest cycle.
For a graph $G$, let $\Phi(G)$ denote its number of perfect matchings,
and let $\delta(G)$ denote its minimum degree.
Following~\cite{Goedgebeur2026}, a graph is \emph{matchable} if it has
a perfect matching, and an edge of a matchable graph is
\emph{inadmissible} if it belongs to no perfect matching.

For $S\subseteq V(G)$, the graph $G-S$ is obtained by deleting the
vertices in $S$ and all edges incident with them.
For $v\in V(G)$, we write $G-v$ for $G-\{v\}$.
Let $N_G(v)=\{u\in V(G):uv\in E(G)\}$
denote the set of neighbors of $v$ in $G$, and let
$d_G(v)=|N_G(v)|$ denote its degree in $G$.
Let $\partial_G(v)$ denote the set of edges of $G$ incident with $v$.
For graphs $G_1$ and $G_2$, the graph $G_1\cup G_2$ has
vertex set $V(G_1)\cup V(G_2)$ and edge set $E(G_1)\cup E(G_2)$.
For a real number $x$, let $\lfloor x\rfloor$ denote the greatest
integer at most $x$, and let $\lceil x\rceil$ denote the least
integer at least $x$.

For each integer $k\ge3$, Zaks~\cite{Zaks1971} constructs
$k$-connected graphs of arbitrarily large order with exactly
$k!$ perfect matchings. Thus connectivity alone does not force
the number of perfect matchings to grow with the order.
These constructions are nonplanar.
The Lov\'asz--Plummer conjecture asserts that the number of
perfect matchings in a cubic bridgeless graph grows exponentially
with its order. Chudnovsky and Seymour~\cite{Chudnovsky2012}
prove the planar case.
Esperet, Kardo\v{s}, King, Kr\'al' and Norine~\cite{Esperet2011} prove
that every cubic bridgeless graph $G$ has at least
$2^{|V(G)|/3656}$ perfect matchings. Thus, in particular, the number of
perfect matchings in a cubic 3-connected graph cannot remain bounded
as its order tends to infinity. The situation changes when regularity
is not required. Goedgebeur, Jooken, Van den Eede and
Zamfirescu~\cite{Goedgebeur2026} construct infinite families of
matchable planar 3-connected graphs with a constant number of perfect
matchings. Their constructions include graphs of girth three and
girth four. They ask whether the same phenomenon occurs for girth
five~\cite[Problem~1]{Goedgebeur2026}.

\Needspace{5\baselineskip}
\begin{samepage}
We give an affirmative answer by constructing the following family.

\begin{The}\label{thm:main}
	There is a positive integer $c$ such that, for every even
	$N\ge28$, there exists a simple planar 3-connected graph $G_N$
	of girth five with $N+168$ vertices and $\Phi(G_N)=c$.
\end{The}
\end{samepage}


To prove Theorem~\ref{thm:main}, we construct an infinite
family of graphs with a unique perfect matching.
We join each graph in this family to a fixed graph by
inadmissible edges.

For every even $n\ge24$, Goedgebeur et al.~\cite[Theorem~2(v)]{Goedgebeur2026}
construct a planar 3-connected graph of order $n$ and minimum
degree four with exactly $768$ perfect matchings.
Goedgebeur et al.~\cite[Section~5]{Goedgebeur2026} also ask whether
infinitely many planar 3-connected graphs of minimum degree five
have a constant number of perfect matchings. We consider this
question for a positive constant.

\begin{The}\label{thm:degree-five}
Let $G$ be a matchable simple planar 3-connected graph with
$\delta(G)\ge5$. Then $\Phi(G)\ge 2^{\lceil(|V(G)|+20)/22\rceil}$.
\end{The}

\begin{Cor}\label{cor:finite-family}
For every positive integer $c$, 
there are only finitely many
  simple planar 3-connected graphs $G$
with $\delta(G)\ge5$ and $1\le\Phi(G)\le c$.
\end{Cor}

%

To prove Theorem~\ref{thm:degree-five}, we first find a set
$W\subseteq V(G)$ such that $G-W$ has a unique perfect
matching. We then use the planarity and 3-connectivity
of $G$ and the condition $\delta(G)\ge5$ to obtain an
upper bound on $|V(G)|$ depending only on $|W|$.

Section~\ref{sec:girth-five} gives the girth-five construction and
proves Theorem~\ref{thm:main}. In Section~\ref{sec:degree-five}, we first prove lemmas
on unique perfect matchings.
We then use these lemmas to prove
Theorem~\ref{thm:degree-five} and
Corollary~\ref{cor:finite-family} for planar 3-connected
graphs of minimum degree at least five.

\section{Planar 3-connected graphs of girth five}\label{sec:girth-five}

\subsection{Perfect matchings and splicing}

\begin{Lem}[Petersen~\cite{Petersen1891}]\label{lem:petersen}
Every bridgeless cubic graph has a perfect matching.
\end{Lem}

For a graph $G$, let $F(G)$ be the graph with vertex set
$V(F(G))=V(G)$ and edge set consisting of all edges of $G$
that belong to at least one perfect matching of $G$.

\begin{Lem}[{\cite[Section~2.2, Fact]{Goedgebeur2026}}]\label{lem:matching-product}
Let $G$ be a matchable graph, and let $C_1,\ldots,C_k$ be
the connected components of $F(G)$.
Then $\Phi(G)=\prod_{i=1}^{k}\Phi(C_i)$.
\end{Lem}

A component is \emph{odd} if it has an odd number of vertices,
and \emph{even} if it has an even number of vertices.
Let $o(G)$ denote the number of odd components of
$G$. A nonempty set $S\subseteq V(G)$ is a \emph{barrier} of a
matchable graph $G$ if $o(G-S)=|S|$. We use the following
property of barriers.

\begin{Lem}[{\cite[proof of Lemma~1.6]{Carvalho2018}}]\label{lem:barrier-even}
Let $S$ be a barrier of a matchable graph $G$, and let $H$
be an even component of $G-S$. Then every edge joining a
vertex of $S$ to a vertex of $H$ is inadmissible in $G$.
\end{Lem}

Let $M$ be a matching of a graph $G$. A cycle of $G$ is
\emph{$M$-alternating} if its edges belong alternately to $M$
and to $E(G)\setminus M$.

\begin{Lem}[{\cite[Lemma~2.2]{Nguyen2018}}]\label{lem:alternating-cycle}
Let $M$ be a perfect matching of a graph $G$. Then $M$ is the
unique perfect matching of $G$ if and only if $G$ contains no
$M$-alternating cycle.
\end{Lem}

Let $G_1$ and $G_2$ be vertex-disjoint graphs, and let
$x\in V(G_1)$ and $y\in V(G_2)$ have the same degree.
Let $\theta:\partial_{G_2}(y)\to\partial_{G_1}(x)$
be a bijection.
The \emph{splice} $G_1(x)\odot G_2(y)$ is the graph
obtained from $G_1-x$ and $G_2-y$ by adding the edge $uv$
for each $yv\in\partial_{G_2}(y)$ with $\theta(yv)=xu$.
The added edges form the \emph{splicing cut}.
When $G_1$ and $G_2$ are plane graphs, we choose $\theta$
so that the added edges can be drawn without crossings.
For a nonempty set $S\subseteq V(G)$ that induces a connected
subgraph, the \emph{contraction} $G/S$ is the graph obtained by
identifying all vertices of $S$ as one vertex, deleting loops
and replacing parallel edges by single edges.

\begin{Lem}\label{lem:splice}
	Let $G$ and $H$ be vertex-disjoint simple 3-connected graphs,
	and let $u\in V(G)$ and $v\in V(H)$ have the same degree.
	Then $G(u)\odot H(v)$ is 3-connected.
\end{Lem}

\begin{proof}
	Set $G^*=G(u)\odot H(v)$, $G'=G-u$, and $H'=H-v$.
	Suppose, to the contrary, that $G^*-S$ is disconnected
	for some $S\subseteq V(G^*)$ with $|S|\le2$.
	Since $G$ and $H$ are 3-connected, both $G'$ and $H'$
	are connected.
	Suppose first that $S\subseteq V(G')$.
	Then $H'$ remains connected in $G^*-S$, so all its
	vertices lie in one component of $G^*-S$.
	Contracting $V(H')$ to a single vertex cannot join
	distinct components. Thus $(G^*-S)/V(H')$ is disconnected.
	By the definition of splicing, this graph is isomorphic
	to $G-S$, where the contraction vertex corresponds to $u$.
	This contradicts the 3-connectivity of $G$.
	Therefore, $S\cap V(H')\ne\varnothing$.
	By symmetry, $S\cap V(G')\ne\varnothing$.
	Since $|S|\le2$, we have $S=\{a,b\}$, where
	$a\in V(G')$ and $b\in V(H')$.
	The graphs $G'-a=G-u-a$ and $H'-b=H-v-b$ are connected
	by the 3-connectivity of $G$ and $H$.
	Since $G$ and $H$ are simple, the edges joining $G'$
	and $H'$ form a matching of at least three edges.
	Deleting $a$ and $b$ removes at most two of these edges,
	so at least one edge remains between $G'-a$ and $H'-b$.
	Hence $G^*-S$ is connected, a contradiction.
	Therefore, $G^*$ is 3-connected.
\end{proof}

\subsection{The construction of \texorpdfstring{$Q_N$ and $L_N$}{Q(N) and L(N)}}
\label{sec:source-construction}

\paragraph{The construction of \texorpdfstring{$Q_N$}{Q(N)}.}
Let $t$ be a nonnegative integer and set $N=28+2t$.
We write $(i,j)$ for the edge joining vertices $i$ and $j$.
Put $U=\{0,1,\ldots,7\}$.
Let $A$ be the graph with vertex set $V(A)=U$ and edge set
$E(A)=\{(0,1),\allowbreak(1,2),\allowbreak(2,3),\allowbreak
(3,4),\allowbreak(4,0),\allowbreak(2,5),\allowbreak
(5,6),\allowbreak(6,7)\}$.

Let $S_t$ be the graph with vertex set
$V(S_t)=\{0,1,\ldots,21+2t\}$ and edge set
$E(S_t)=E(A)
\cup\{(i,i+1):8\le i\le20+2t\}
\cup\{(0,10),\allowbreak(1,8),\allowbreak(3,14),\allowbreak
(4,12),\allowbreak(5,16),\allowbreak(6,18),\allowbreak
(7,8),\allowbreak(7,20)\}
\cup\{(9+2j,22+2j):0\le j\le t-1\}$.
The graphs $S_0$ and $S_1$ are shown in
Figure~\ref{fig:q30-construction}(a) and (b), respectively.

Put $p=21+2t$ and $p_i=7+2t+2i$ for $1\le i\le6$.
In $S_t$, the vertex $p$ has degree one, each $p_i$ has
degree two, and every other vertex has degree three.
Put $z=N-6$, $q_1=N-5$, $q_2=N-2$, $q_3=N-1$,
$q_4=N-3$, and $q_5=N-4$.
Let $Q_N$ be the graph with vertex set
$V(Q_N)=V(S_t)\cup\{z,q_1,\ldots,q_5\}$ and edge set
$E(Q_N)=E(S_t)
\cup\{q_iq_{i+1}:1\le i\le4\}
\cup\{q_5q_1,zq_1,pp_1,pz,zp_2\}
\cup\{p_iq_{i-1}:3\le i\le6\}$.
Figure~\ref{fig:q30-construction}(c) shows the added
vertices and edges.
Thus $|V(Q_N)|=N$, and $Q_N$ is cubic.

A \emph{fullerene graph} is a cubic 3-connected plane graph
whose faces have length five or six.
The graph $Q_N$ is the fullerene with parameter $k=t+14$
in the construction of Yang, Zhang and
Lin~\cite[proof of Theorem~4.4]{Yang2015}.
Hence $Q_N$ is planar and 3-connected and has girth five.
Figure~\ref{fig:q30-construction}(d) shows $Q_{30}$.

\begin{figure}[H]
\centering
\includegraphics[width=.90\textwidth]{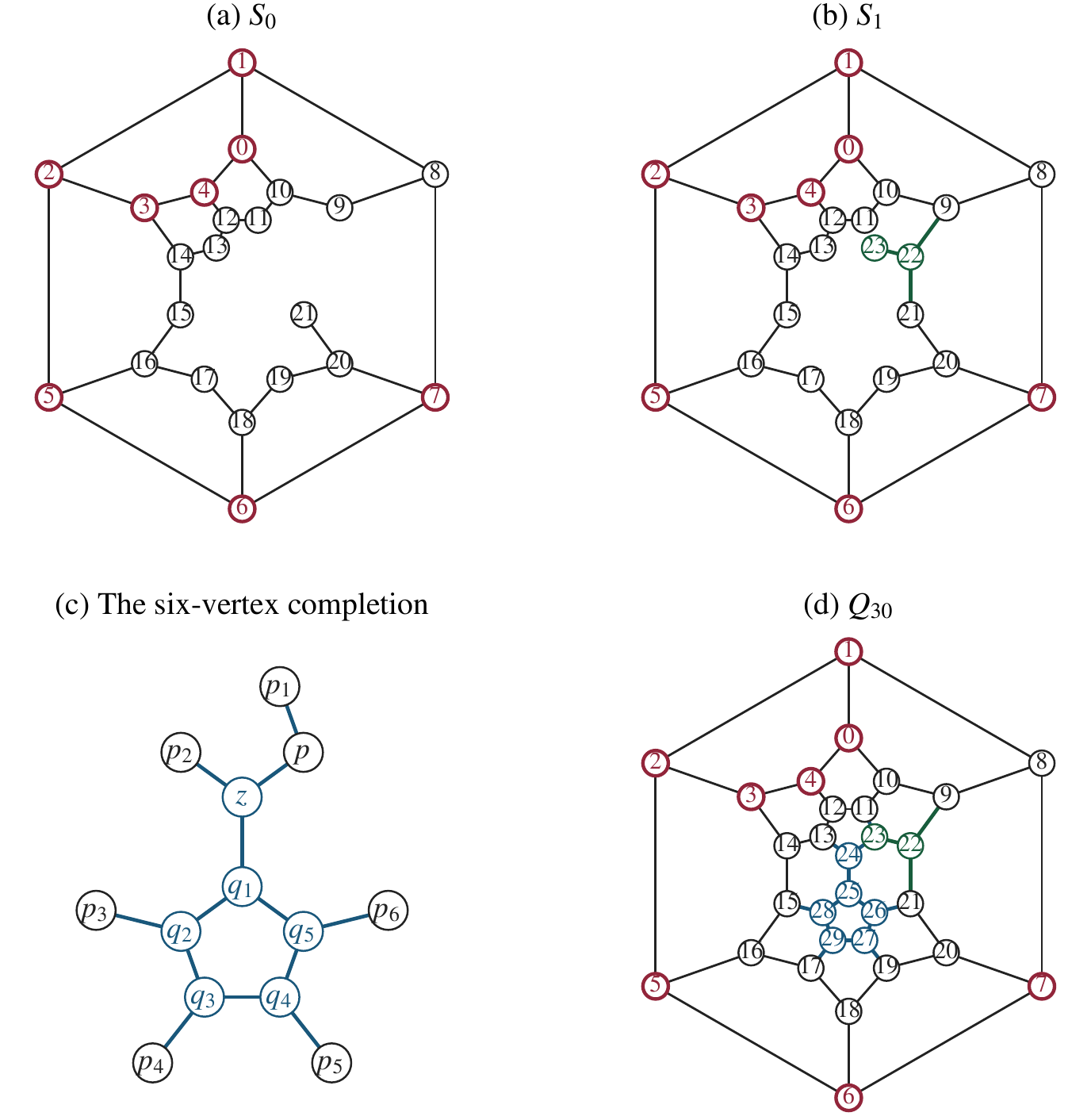}
\caption{The construction of $Q_{30}$. Panels (a), (b) and (d)
show $S_0$, $S_1$ and $Q_{30}$ in the same plane drawing.
Red vertices with thick outlines form $U$. Green vertices and edges are added in
the operation from $S_0$ to $S_1$; blue vertices and edges are
added when completing $S_1$. Panel (c) enlarges the latter
addition, using the general notation $p,p_1,\ldots,p_6,q_1,\ldots,q_5,z$;
all existing edges of $S_1$ are omitted from that panel.}
\label{fig:q30-construction}
\end{figure}


\paragraph{The construction of \texorpdfstring{$L_N$}{L(N)}.}
Let $L_N=Q_N-U$. Then $V(L_N)=\{8,9,\ldots,N-1\}$ and
$E(L_N)=\{(i,j)\in E(Q_N):i,j\notin U\}$.
Thus $|V(L_N)|=N-8=20+2t$.
The graph $L_{28}$ has $20$ vertices and has the smallest
order in this family.
As $t$ increases, $|V(L_N)|$ tends to infinity.
Figure~\ref{fig:source} shows $L_{30}=Q_{30}-U$.

\begin{figure}[H]
\centering
\includegraphics[width=.55\textwidth]{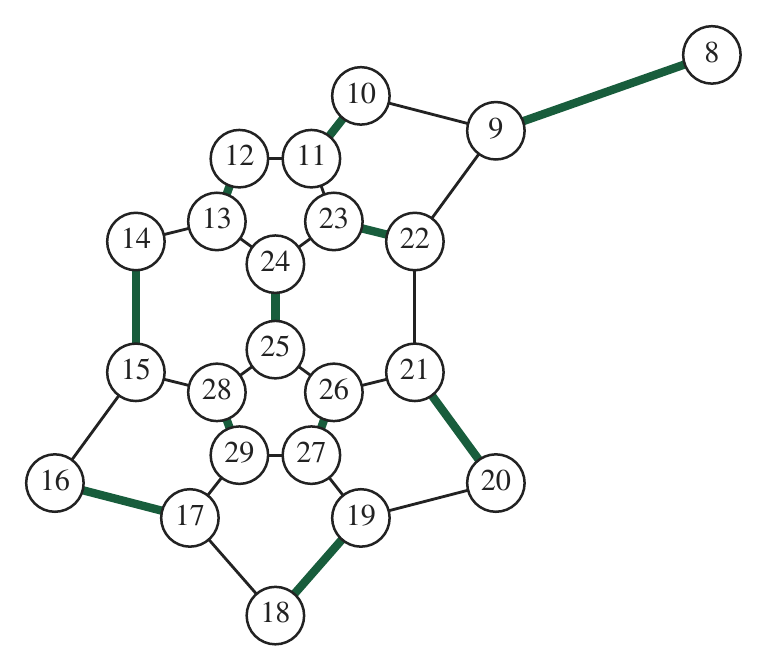}
\caption{The graph $L_{30}=Q_{30}-U$, with the vertex labels
inherited from Figure~\ref{fig:q30-construction}(d).
The thick green edges form the
unique perfect matching of $L_{30}$.}
\label{fig:source}
\end{figure}

\begin{Lem}\label{lem:unique-LN}
	For every even integer $N\ge28$, the graph $L_N$ has
	a unique perfect matching.
\end{Lem}

	\begin{proof}
		Let $N=28+2t$, where $t$ is a nonnegative integer.
		By construction,
		$M_N=\{(2k,2k+1):4\le k\le t+10\}
		\cup\{zq_1,q_2q_3,q_4q_5\}$
		is a perfect matching of $L_N$.
		Suppose that $L_N$ contains an $M_N$-alternating cycle $C$.
		If $C$ contains a vertex of $S_t-U$, then it contains
		a matching edge $(2k,2k+1)$ with $4\le k\le t+10$.
		Let $k$ be the smallest such integer.
		Let $(2k,w)$ be the edge of $C$ incident with $2k$
		that does not belong to $M_N$.
		By construction, $8\le w<2k$.
		Thus $w\in\{2i,2i+1\}$ for some integer $i$
		with $4\le i<k$.
		Since $C$ is $M_N$-alternating, $(2i,2i+1)\in E(C)$.
		This contradicts the minimality of $k$.
		Hence $V(C)\subseteq\{z,q_1,q_2,q_3,q_4,q_5\}$.
		These vertices induce the cycle $q_1q_2q_3q_4q_5q_1$
		together with the edge $zq_1$.
		The only cycle in this subgraph has length five,
		but an $M_N$-alternating cycle has even length,
		a contradiction.
		Thus, $L_N$ does not contain any $M_N$-alternating cycle.
		By Lemma~\ref{lem:alternating-cycle}, $M_N$ is
		the unique perfect matching of $L_N$.
	\end{proof}

\FloatBarrier
\subsection{The construction of \texorpdfstring{$B$ and $G_N$}{B and G(N)}}
\label{sec:full-construction}

\paragraph{The construction of \texorpdfstring{$B$}{B}.}
Let $H$ be the bipartite graph with partite sets
$X=\{a_{2i},b_{2i+1}:0\le i\le7\}$ and
$Y=\{a_{2i+1},b_{2i}:0\le i\le7\}$ and edge set
$E(H)=\{a_i a_{i+1},b_i b_{i+1},a_i b_i:0\le i\le15\}$,
where $a_{16}=a_0$ and $b_{16}=b_0$.
Figure~\ref{fig:H-D-B}(a) shows a plane drawing of $H$.
For each $0\le i\le7$, let $D_i$ be a dodecahedral graph
and let $z_i\in V(D_i)$, as shown in
Figure~\ref{fig:H-D-B}(b).
Assume that $H,D_0,\ldots,D_7$ are pairwise vertex-disjoint.
Let $B$ be the graph obtained by splicing $H$ at $a_{2i+1}$
with $D_i$ at $z_i$ for each $i$, choosing the bijections
to respect the plane drawings.

Clearly, $B$ is a simple cubic planar graph of order $176$
and girth $5$; see Figure~\ref{fig:H-D-B}(c).
Since $H$ and each $D_i$ are 3-connected,
Lemma~\ref{lem:splice} implies that $B$ is 3-connected
and hence bridgeless.
Let $c=\Phi(B)$, which is a positive integer
by Lemma~\ref{lem:petersen}.

\begin{figure}[!htbp]
\centering
\includegraphics[height=.72\textheight,width=\textwidth,keepaspectratio,
trim=0 65 0 0,clip]{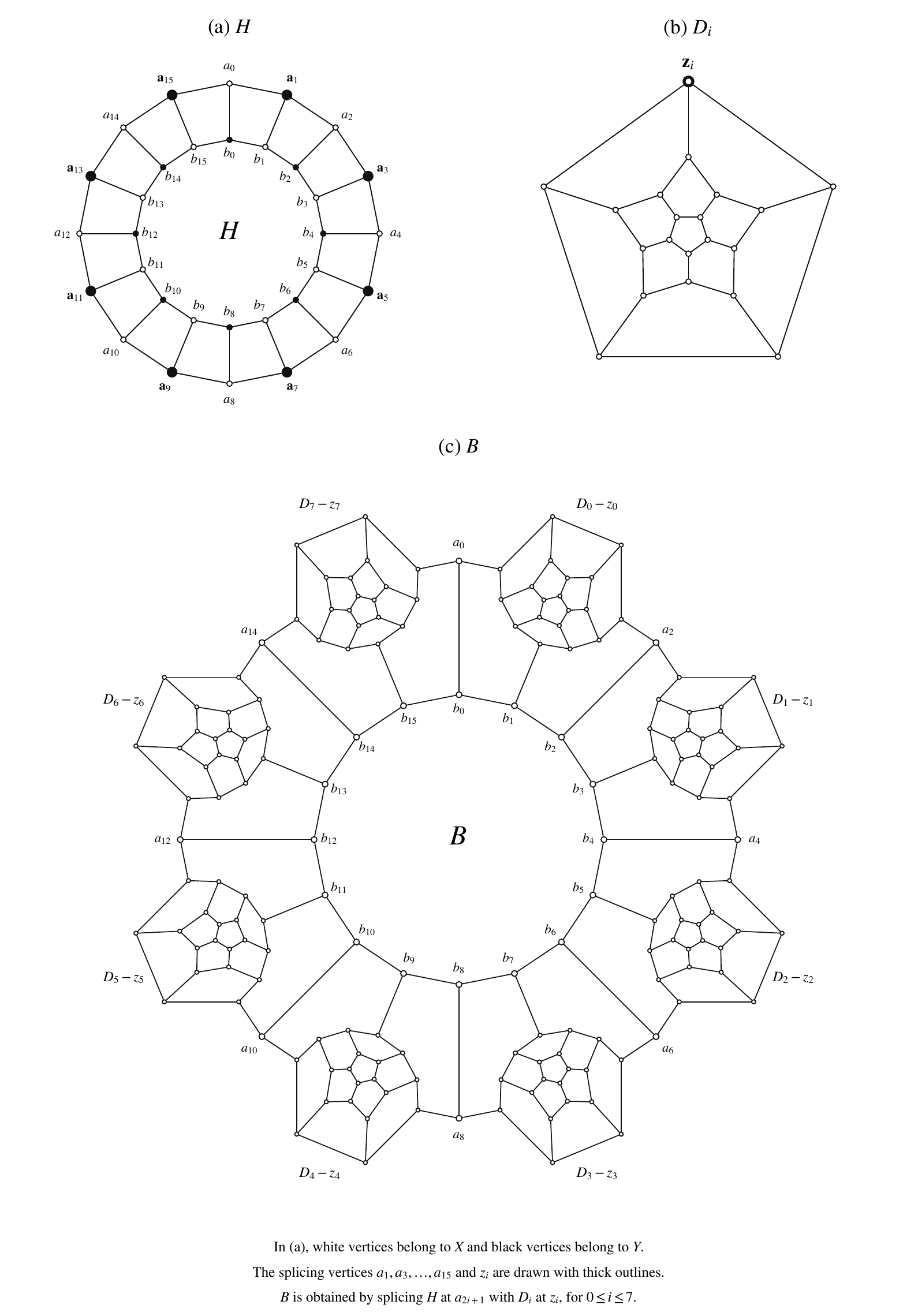}
\caption{The graphs $H$, $D_i$ and $B$. In (a), white and black
vertices indicate the partite sets $X$ and $Y$, respectively.
The splicing vertices $a_1,a_3,\ldots,a_{15}$ in (a) and $z_i$
in (b) are drawn with thick outlines.}
\label{fig:H-D-B}
\end{figure}

\Needspace{8\baselineskip}
\paragraph{The construction of \texorpdfstring{$G_N$}{G(N)}.}
Let $N\ge28$ be even, and assume that $B$ and $L_N$
are vertex-disjoint. Put
$F_N=\{(a_0,20),\allowbreak(a_2,18),\allowbreak
(a_4,16),\allowbreak(a_6,14),\allowbreak(a_8,12),\allowbreak
(a_{10},10),\allowbreak(a_{12},8),\allowbreak(a_{14},8)\}$.
Let $G_N$ be the graph with vertex set
$V(G_N)=V(B)\cup V(L_N)$ and edge set
$E(G_N)=E(B)\cup E(L_N)\cup F_N$.
Then $|V(G_N)|=N+168$.
By construction, $G_N$ is planar and has girth $5$.
Figure~\ref{fig:complete} shows a plane drawing of $G_{28}$.
\begin{figure}[!htbp]
\centering
\includegraphics[height=.68\textheight,width=\textwidth,keepaspectratio]{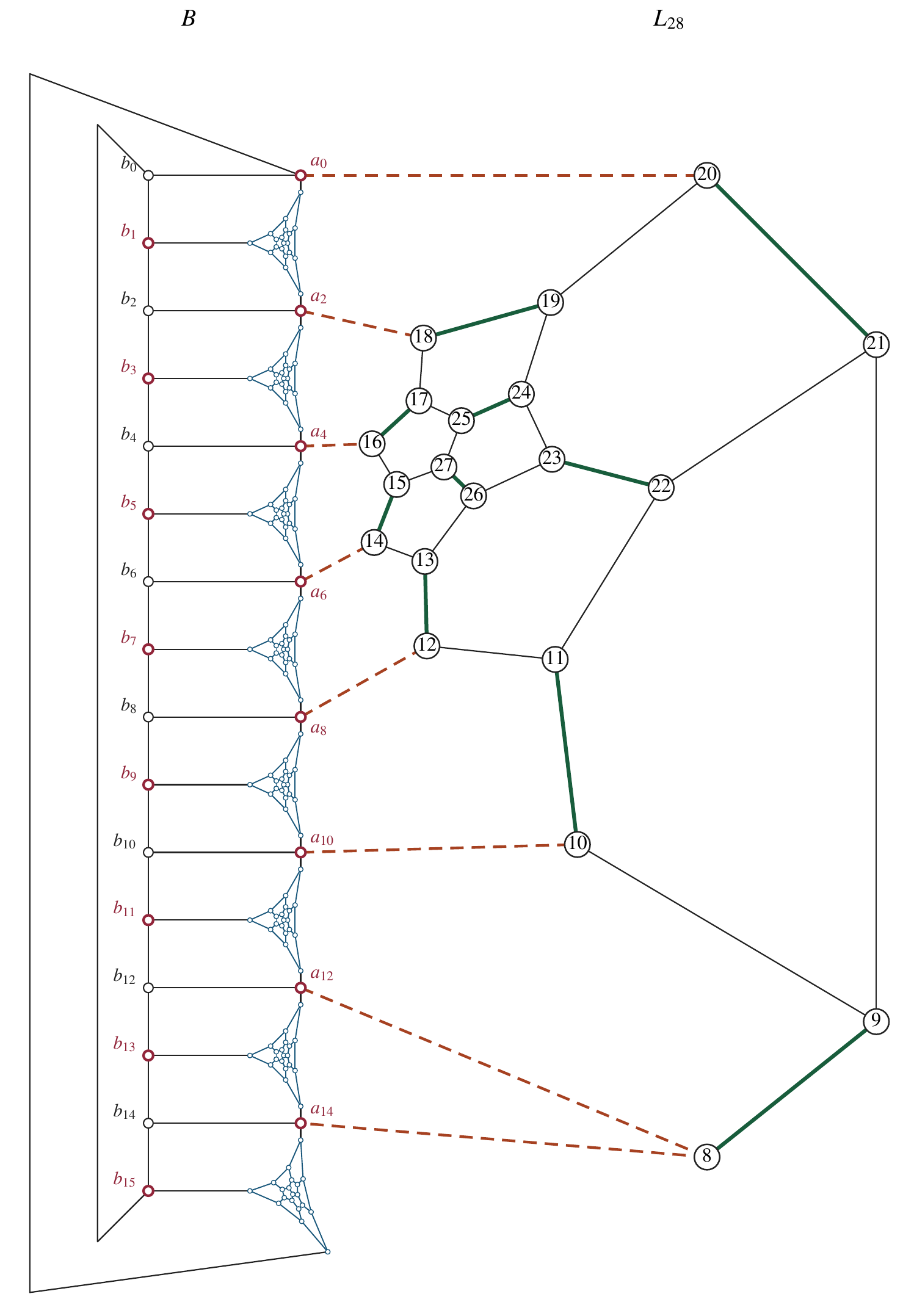}
\caption{The graph $G_{28}$, of order $196$. The fixed graph $B$
is on the left, and $L_{28}$ is on the right. All vertices and
edges are shown. Red circles with thick outlines indicate the
sixteen vertices of $X$; the eight blue subgraphs are the graphs
$D_i-z_i$. Dashed orange edges form $F_{28}$, and thick green edges form the unique
perfect matching of $L_{28}$. The two bends on each of the
leftmost edges are not vertices.}
\label{fig:complete}
\end{figure}

\FloatBarrier
\Needspace{10\baselineskip}
\subsection{\texorpdfstring{Proof of Theorem~\ref{thm:main}}{Proof of Theorem 1}}

\begin{proof}[Proof of Theorem~\ref{thm:main}]
	Let $N\ge28$ be even, and let $G_N$ be the graph constructed
	in Section~\ref{sec:full-construction}.
	Recall that $c=\Phi(B)$ is a positive integer.
	By construction, $G_N$ is a simple planar graph of order
	$N+168$ and girth $5$.
	
	\Needspace{5\baselineskip}
	\begin{Cla}\label{claim:connectivity}
		The graph $G_N$ is 3-connected.
	\end{Cla}
%
%

\begin{proof}
	Let $S\subseteq V(G_N)$ with $|S|\le2$, and put
	$W=S\cap V(L_N)$ and $Z=S\cap V(B)$.
	Since $B$ is 3-connected, $B-Z$ is connected.
	Let $C$ be a component of $L_N-W$.
	Since $L_N=Q_N-U$, every neighbor of $C$ outside $C$
	in $Q_N$ belongs to $U\cup W$.
	Suppose that $C$ has at most $2-|W|$ neighbors in $U$.
	Deleting $W$ and the neighbors of $C$ in $U$ removes
	at most two vertices from $Q_N$ and separates $C$
	from a vertex of $U$, since $|U|=8$.
	This contradicts the 3-connectivity of $Q_N$.
	Thus $C$ has at least $3-|W|$ neighbors in $U$,
	so at least $3-|W|$ edges of $Q_N$ join $C$ to $U$.
	By construction, every edge of $Q_N$ joining $C$ to $U$
	is replaced by a distinct edge of $F_N$ joining $C$ to $B$.
	The edges of $F_N$ have distinct incident vertices in $B$.
	Hence $C$ has at least $3-|W|$ neighbors in $B$.
	Since $|Z|\le2-|W|<3-|W|$, at least one neighbor
	of $C$ in $B$ does not belong to $Z$.
	Therefore, $C$ is joined to $B-Z$ in $G_N-S$.
	Every component of $L_N-W$ is therefore joined
	to the connected graph $B-Z$ in $G_N-S$.
	Hence $G_N-S$ is connected, and $G_N$ is 3-connected.
\end{proof}

\begin{Cla}\label{claim:matching}
	Every edge of $F_N$ is inadmissible.
\end{Cla}

\begin{proof}
	Let $M_B$ be a perfect matching of $B$, and let $M_N$
	be the unique perfect matching of $L_N$.
	Since $V(G_N)=V(B)\cup V(L_N)$ and the two vertex sets
	are disjoint, $M_B\cup M_N$ is a perfect matching of $G_N$.
	Thus $G_N$ is matchable.
	By construction, the components of $G_N-X$ are the eight
	graphs $D_i-z_i$, the eight isolated vertices $b_{2i}$,
	where $0\le i\le7$, and the graph $L_N$.
	Each $D_i-z_i$ has $19$ vertices, while $L_N$
	has even order $N-8$.
	Hence $o(G_N-X)=16=|X|$, so $X$ is a barrier of $G_N$.
	Every edge of $F_N$ joins a vertex of $X$ to a vertex
	of the even component $L_N$ of $G_N-X$.
	Therefore, every edge of $F_N$ is inadmissible
	by Lemma~\ref{lem:barrier-even}.
\end{proof}

	By Claim~\ref{claim:matching}, $F(G_N)=F(B)\cup F(L_N)$.
	Lemma~\ref{lem:matching-product}, applied to $G_N$, $B$
	and $L_N$, gives $\Phi(G_N)=\Phi(B)\Phi(L_N)$.
	By Lemma~\ref{lem:unique-LN}, $\Phi(L_N)=1$.
	Therefore, $\Phi(G_N)=\Phi(B)=c$.
	Together with Claim~\ref{claim:connectivity},
	this proves the theorem.
\end{proof}

\FloatBarrier
\section{Planar 3-connected graphs of minimum degree five}\label{sec:degree-five}

\subsection{Reduction to a unique perfect matching}
\label{sec:unique-deletion}

We first reduce the minimum-degree-five problem to graphs with a
unique perfect matching.
 We call the graph with no vertices the
\emph{empty graph}. Its unique perfect matching is the empty set of edges.

\Needspace{7\baselineskip}
\begin{Lem}\label{lem:halving}
	Let $G$ be a graph with exactly $p$ perfect matchings,
	where $p\ge1$.
	Then there exists a set $W\subseteq V(G)$ such that
	$|W|\le2\lfloor\log_2p\rfloor$ and $G-W$ has exactly
	one perfect matching.
\end{Lem}
\begin{proof}
We use induction on $p=\Phi(G)$. If $p=1$, then $G$ has a
unique perfect matching. Taking $W=\varnothing$ gives
$G-W=G$ and $|W|=0=2\lfloor\log_2p\rfloor$.
Suppose that $p>1$. Then $G$ has at least two distinct
perfect matchings.
Hence some vertex $v$ is incident with at least two distinct edges,
each of which belongs to a perfect matching of $G$.
Among the edges incident with $v$ that belong to a perfect
matching, choose $uv$ so that $\Phi(G-u-v)$ is as small as possible.
Put $G'=G-u-v$ and $q=\Phi(G')$.
Since $uv$ belongs to a perfect matching of $G$, the graph $G'$
has a perfect matching, so $q\ge1$.

For each edge $vw$ incident with $v$, the perfect matchings of $G$
containing $vw$ are precisely the sets $\{vw\}\cup M$, where
$M$ is a perfect matching of $G-v-w$.
Thus $uv$ belongs to exactly $q$ perfect matchings of $G$.
By the choice of $uv$, another edge incident with $v$ belongs
to at least $q$ perfect matchings of $G$.
No perfect matching contains both edges, since they are incident
with the same vertex $v$. Therefore $p\ge2q$, and hence
$1\le\Phi(G')=q\le p/2<p$.
By the induction hypothesis applied to $G'$, there exists
$W'\subseteq V(G')$ with $|W'|\le2\lfloor\log_2q\rfloor$
such that $G'-W'$ has a unique perfect matching.
Let $W=W'\cup\{u,v\}$. Since $G-W=G'-W'$, the graph $G-W$
has a unique perfect matching.
Also, $u,v\notin W'$, so $|W|=|W'|+2$.
Since $q\le p/2$, we have
$\lfloor\log_2q\rfloor+1\le\lfloor\log_2p\rfloor$.
Therefore $|W|\le2\lfloor\log_2q\rfloor+2
\le2\lfloor\log_2p\rfloor$.
\end{proof}

\subsection{An order bound for plane graphs}
\label{sec:plane-order-bound}

We use the following classical theorem in the next argument.

\begin{Lem}[Kotzig; see~\cite{Nenashev2014}]\label{lem:kotzig}
If a connected graph has a unique perfect matching, then some
edge of that matching is a bridge.
\end{Lem}

We next prove an order bound for plane graphs.

Let $H$ be a connected simple plane graph, and let
$f_1,\ldots,f_h$ be distinct faces of $H$, where $h=0$ is allowed.
For each $1\le i\le h$, let $C_i$ denote the boundary of $f_i$,
and assume that $C_1,\ldots,C_h$ are pairwise vertex-disjoint cycles.
Put $\partial V(H)=\bigcup_{i=1}^{h}V(C_i)$ and
$b=|\partial V(H)|$.
Let $\widehat H$ be the plane graph obtained from $H$
by adding one new vertex in each face $f_i$ and joining
it to every vertex of $C_i$, with all these new edges
drawn inside $f_i$.
If $h=0$, then $\partial V(H)=\varnothing$ and
$\widehat H=H$.
Let $Z\subseteq V(H)\setminus\partial V(H)$, and put
$I=H-(\partial V(H)\cup Z)$, $r=|Z|$, and $q=|V(I)|$.
Thus $|V(H)|=b+r+q$.

Figure~\ref{fig:boundary-graphs}(a)--(c) illustrates
the above definitions for one particular choice of $H$.
In this example, $H$ consists of the cycle $C=xabycdx$,
two additional vertices $u,v$, and the edges
$ux,ua,ub,uy,vx,vy,vc,vd,uv$.
In the plane drawing shown in Figure~\ref{fig:boundary-graphs}(a),
let $f_1$ be the face bounded by $C$, and take
$h=1$, $C_1=C$, and $Z=\varnothing$.
Then $\partial V(H)=V(C)=\{x,a,b,y,c,d\}$,
$b=6$, and $r=0$.
By the definition of $\widehat H$, we add a new vertex $w$
in $f_1$ and join it to each of the six
vertices of $C$, retaining all vertices and edges of $H$.
This gives the graph $\widehat H$ shown in
Figure~\ref{fig:boundary-graphs}(b), where the six added
edges are dashed.
Since $Z=\varnothing$, we have $I=H-V(C)$.
Deleting the six vertices of $C$ from $H$ leaves only
the vertices $u,v$ and the edge $uv$; see
Figure~\ref{fig:boundary-graphs}(c).
Thus $q=2$, and $\{uv\}$ is the unique perfect matching
of $I$. Moreover, both $u$ and $v$ have degree five in $H$.

\begin{figure}[!htbp]
\centering
\input{figures/lemma11_selected_faces_GH.tex}
\caption{A configuration with $h=1$, $b=6$, $r=0$, and $q=2$.
The face $f_1$ has boundary $C=xabycdx$, whose edges are drawn in gray.
In (b), adding the vertex $w$ and the dashed edges to $H$ gives
$\widehat H$. In (c), $I=H-V(C)$ consists of the thick edge $uv$.
Both $u$ and $v$ have degree five in $H$.
This example illustrates the definitions of $H$, $\widehat H$, and $I$.}
\label{fig:boundary-graphs}
\end{figure}
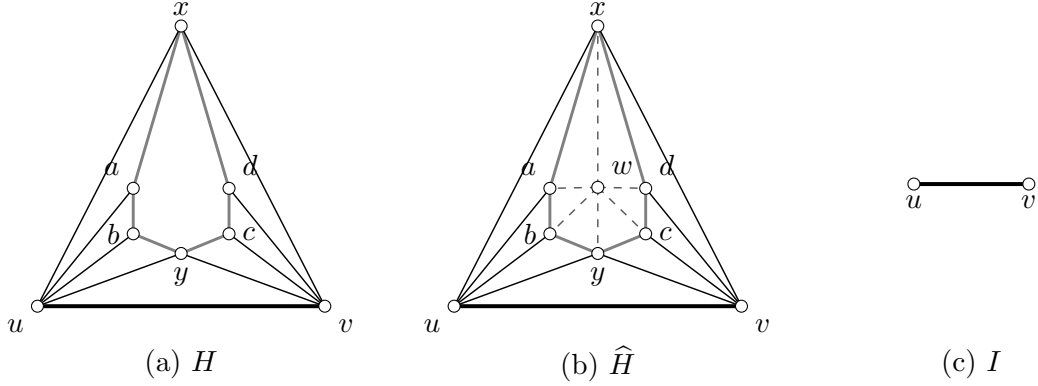

We use the following edge bound: every simple planar bipartite graph
$G$ with $|V(G)|\ge3$ satisfies $|E(G)|\le2|V(G)|-4$
\cite[Corollary~1.14]{Thomas2010}.

\begin{Lem}[{\cite[Section~12.6]{Erickson2023}}]
\label{lem:facial-connectivity}
Let $G$ be a connected simple plane graph with $|V(G)|\ge4$.
Then $G$ is 3-connected if and only if the boundary of every
face is a cycle and the boundaries of any two distinct faces
either have at most one common vertex or have only a single
edge and its two incident vertices in common.
\end{Lem}

\begin{Lem}[{\cite[Theorem~3.2]{Thomas2010}}]
	\label{lem:induced-facial-cycle}
Let $G$ be a simple 3-connected plane graph, and let $f$
be a face of $G$.
Then the boundary of $f$ is a cycle $C$.
If $u,v\in V(C)$ and $uv\notin E(C)$, then $uv\notin E(G)$.
\end{Lem}

\Needspace{8\baselineskip}
\begin{Lem}\label{lem:boundary-unique}
Let $\widehat H$ be a
	simple and 3-connected plane graph.
	If the graph $I=H-(\partial V(H)\cup Z)$ has at least
	two vertices and a unique perfect matching, and every
	vertex of $I$ has degree at least $5$ in $H$, then
	$q\le2b+10h+10r-20$.
\end{Lem}

\begin{proof}
We use strong induction on the positive even integer $q$.
Suppose first that $q=2$, so $I$ consists of an edge $uv$.
Let $G$ be the bipartite graph whose two parts are
$\partial V(H)\cup Z$ and the set consisting of $u,v$
and the $h$ vertices added in the construction of $\widehat H$.
Its edge set consists of all edges of $\widehat H$
joining these two parts.
Then $G$ is simple and planar, with $b+r+h+2$ vertices.
The $h$ added vertices are joined to $\partial V(H)$
by exactly $b$ edges, since each boundary vertex is adjacent
to exactly one of them.
Since $u,v$ have degree at least five in $H$ and $uv$
is the only edge of $I$, each has at least four neighbors
in $\partial V(H)\cup Z$.
Thus $|E(G)|\ge b+8$.
The planar bipartite edge bound gives
$b+8\le2(b+r+h+2)-4$, so $b+2h+2r\ge8$.
If $h=r=0$, then $H=I$ consists of the edge $uv$,
and both vertices have degree one, a contradiction.
Hence $h+r\ge1$, and therefore
$2b+10h+10r-20
=2(b+2h+2r)+6(h+r)-20
\ge16+6-20=2=q$.

Now suppose that $q\ge4$.
Every face of $H$ other than $f_1,\ldots,f_h$ is also
a face of $\widehat H$.
Since $\widehat H$ is simple and 3-connected,
Lemma~\ref{lem:induced-facial-cycle} implies that
the boundary of every face of $\widehat H$ is a cycle.
Consider each face $f$ of $H$ other than $f_1,\ldots,f_h$.
If the boundary of $f$ has length at least four and
contains a vertex of $\partial V(H)\cup Z$, choose
$v_1\in\partial V(H)\cup Z$ on the boundary of $f$
and write the boundary cycle as $v_1v_2\cdots v_\ell v_1$.
For $3\le i\le\ell-1$, the edge $v_1v_i$ does not
belong to the boundary cycle of $f$.
By Lemma~\ref{lem:induced-facial-cycle} applied to
$\widehat H$, we have $v_1v_i\notin E(\widehat H)$.
Since $E(H)\subseteq E(\widehat H)$, we also have
$v_1v_i\notin E(H)$.
Add the edges $v_1v_i$ for $3\le i\le\ell-1$
inside $f$ without crossings.
The face $f$ is divided into triangular faces with
boundary cycles $v_1v_iv_{i+1}v_1$, where $2\le i\le\ell-1$.


Let $M$ be the unique perfect matching of $I$.
Each component of $I$ has a unique perfect matching, so
Lemma~\ref{lem:kotzig} gives a bridge $uv\in M$ of $I$.
If $uv$ lies on the boundary of $f_i$ for some $i$,
then $u,v\in V(C_i)\subseteq\partial V(H)$.
This is impossible, since $u,v\in V(I)$.
Hence neither face of $H$ incident with $uv$ is among
$f_1,\ldots,f_h$.
Let $f$ be a face of $H$ incident with $uv$.
Since $f$ is also a face of $\widehat H$,
Lemma~\ref{lem:induced-facial-cycle} applied to $\widehat H$
implies that the boundary of $f$ is a cycle.
Suppose that all vertices on the boundary of $f$
belong to $I$.
Since $I=H-(\partial V(H)\cup Z)$, all edges on
the boundary of $f$ also belong to $I$.
Thus the boundary of $f$ is a cycle of $I$ containing $uv$,
contrary to the fact that $uv$ is a bridge of $I$.
Therefore, the boundary of $f$ contains a vertex of
$\partial V(H)\cup Z$.
By the preceding edge additions, the two faces of $H$
incident with $uv$ are triangular.
Denote the boundary triangles by $uvx$ and $uvy$.
Each triangle contains a vertex of $\partial V(H)\cup Z$.
Since $u,v\in V(I)$, we have $x,y\in\partial V(H)\cup Z$.
If $x=y$, the two distinct faces of $\widehat H$
bounded by $uvx$ and $uvy$ have all three boundary edges
in common, contrary to Lemma~\ref{lem:facial-connectivity}.
Hence $x\ne y$.

Let $G$ be the plane graph obtained from $H$ by deleting
the edge $uv$, keeping the vertices $u,v$.
The two faces bounded by $uvx$ and $uvy$ become one face
of $G$ bounded by $xuyvx$.
Split the vertices of $\{x,y\}\cap\partial V(H)$,
with $x$ before $y$ if $x,y\in\partial V(H)$.
For each $w\in\{x,y\}\cap\partial V(H)$, let $C_k$
be the boundary cycle containing $w$ in $H$.
Use the order of the edges incident with $w$ in the
drawing of $H$, with the two edges of $C_k$ incident
with $w$ first and last.
Write $N_G(w)=\{w_1,\ldots,w_d\}$ according to this order.
Since $uvw$ bounds a face of $H$, the edges $wu,wv$
have no other edge between them along that face.
Choose the direction so that $w_j=u$ and $w_{j+1}=v$
for some $1\le j<d$.
Let $G'$ be the plane graph obtained from $G$
by replacing $w$ with two nonadjacent vertices $w_u,w_v$,
where $N_{G'}(w_u)=\{w_1,\ldots,w_j\}$ and
$N_{G'}(w_v)=\{w_{j+1},\ldots,w_d\}$;
see Figure~\ref{fig:boundary-step}(b).
Keep the order $w_1,\ldots,w_j$ around $w_u$ and
$w_{j+1},\ldots,w_d$ around $w_v$.
At every vertex of $G$ other than $w$,
keep the order of the incident edges unchanged.
All edges of $G$ not incident with $w$ remain unchanged.
After each split, replace $G$ by $G'$.
After splitting every vertex of
$\{x,y\}\cap\partial V(H)$, put $H'=G$.

\begin{figure}[!htbp]
\centering
\input{figures/lemma_boundary_schematic_GH.tex}
\caption{A schematic illustration of the operation in
Lemma~\ref{lem:boundary-unique} for $x,y\in V(C_k)$ and $a=2$.
Panels (a), (c), and (d) show the relevant parts of $H,H_1,H_2$,
respectively; ellipses indicate omitted vertices and edges.
Panel (b) shows the neighbors of $w$ before and after splitting,
with $w_j=u$ and $w_{j+1}=v$.
Gray lines belong to the boundary cycles of $f_k,f_{1,1},f_{2,1}$.
The edge $uv\in M$ is thick, and the faces $uvx,uvy$ are shaded.
The internal structures of $I_1,I_2$ are omitted.}
\label{fig:boundary-step}
\end{figure}


For each $w\in\{x,y\}\cap\partial V(H)$,
the vertices $w_u,w_v$ are adjacent to $u,v$, respectively.
Since $H$ is connected, every component of $H'$
contains $u$ or $v$.
Let $H_1,\ldots,H_a$ be the components of $H'$,
where $1\le a\le2$.
After splitting each $w\in\{x,y\}\cap\partial V(H)$
into $w_u,w_v$, the edges of $C_1,\ldots,C_h$ and
$xuyvx$ form pairwise vertex-disjoint cycles.
These cycles bound faces of the graphs $H_i$.
For each $1\le i\le a$, denote these faces in $H_i$
by $f_{i,1},\ldots,f_{i,h_i}$;
see Figure~\ref{fig:boundary-step}(c) and (d).
Let $\partial V(H_i)$ be the set of vertices on the
boundaries of $f_{i,1},\ldots,f_{i,h_i}$.
Put $Z_i=(Z\setminus\{x,y\})\cap V(H_i)$ and
$I_i=H_i-(\partial V(H_i)\cup Z_i)$.
Write $b_i=|\partial V(H_i)|$, $r_i=|Z_i|$,
and $q_i=|V(I_i)|$.
Let $\widehat H_i$ be the plane graph obtained from $H_i$
by adding one vertex in each face $f_{i,j}$
and joining it to every vertex of its boundary,
with all new edges drawn inside $f_{i,j}$.

\Needspace{8\baselineskip}
Each $H_i$ contains $u$ or $v$, and both vertices
belong to the boundaries of the faces $f_{i,j}$.
Hence $h_i\ge1$.
Since $d_{H'}(u)=d_H(u)-1\ge4$ and
$d_{H'}(v)=d_H(v)-1\ge4$,
we also have $|V(H_i)|\ge5$.
The graphs $I_1,\ldots,I_a$ are pairwise vertex-disjoint,
and $\bigcup_{i=1}^{a}I_i=I-u-v$.
Hence $\sum_{i=1}^{a}q_i=q-2$.
For each $i$, the set $M\cap E(I_i)$ is a perfect matching of $I_i$.
If $I_i$ has another perfect matching, replacing $M\cap E(I_i)$
by that matching gives a second perfect matching of $I$.
Therefore, each $I_i$ has a unique perfect matching.

\FloatBarrier
\Needspace{6\baselineskip}

\begin{Cla}\label{claim:boundary-connectivity}
	For each $i$, the graph $\widehat H_i$ is simple and 3-connected.
\end{Cla}

\begin{proof}
	By construction, $\widehat H_i$ is a simple connected plane
	graph with at least five vertices.
	For each face of $H_i$ other than $f_{i,1},\ldots,f_{i,h_i}$,
	its boundary is obtained from a face boundary of $\widehat H$
	by splitting vertices.
	At each vertex $w$ that is split, the two edges of that
	boundary incident with $w$ are both attached to $w_u$
	or both to $w_v$.
	Thus, by Lemma~\ref{lem:induced-facial-cycle} applied to
	$\widehat H$, the boundary in $H_i$ is a cycle $C$,
	and among the vertices of $C$, the only edges of $H_i$
	are the edges of $C$.
	All other faces of $\widehat H_i$ are triangular.

	Consider two distinct faces of $H_i$ other than
	$f_{i,1},\ldots,f_{i,h_i}$.
	By Lemma~\ref{lem:facial-connectivity} applied to $\widehat H$,
	before splitting, their boundaries have at most one common
	vertex or have only an edge and its two incident vertices
	in common.
	Splitting does not increase the number of common vertices,
	and every common boundary edge remains on both boundaries.

	Let $z$ be the vertex added inside $f_{i,j}$.
	Each triangular face inside $f_{i,j}$ has boundary
	$zv_1v_2z$ for an edge $v_1v_2$ on the boundary of $f_{i,j}$.
	For two distinct triangular faces inside $f_{i,j}$,
	if the two boundary edges have no common vertex,
	the triangles have only $z$ in common.
	If the two boundary edges have a common vertex $v_1$,
	the triangles have only the edge $zv_1$ and its two
	incident vertices in common.
	Triangular faces inside different faces $f_{i,j}$
	have no common vertex, since the boundary cycles are
	pairwise vertex-disjoint and the added vertices are distinct.

	Finally, let $C$ bound a face of $H_i$ other than
	$f_{i,1},\ldots,f_{i,h_i}$, and let $zv_1v_2z$
	bound a triangular face with added vertex $z$.
	Since $z\notin V(H_i)$, the only possible common
	vertices are $v_1,v_2$.
	If both belong to $C$, then $v_1v_2\in E(C)$,
	since $v_1v_2\in E(H_i)$ and the only edges of $H_i$
	among the vertices of $C$ are the edges of $C$.
	Hence the two boundaries have at most one common vertex
	or have only the edge $v_1v_2$ and its two incident
	vertices in common.
	Therefore, Lemma~\ref{lem:facial-connectivity}
	implies that $\widehat H_i$ is 3-connected.
\end{proof}

\Needspace{6\baselineskip}
\begin{Cla}\label{claim:boundary-degree}
For each $i$, every vertex of $I_i$ has degree at least five in $H_i$.
\end{Cla}
\begin{proof}
Let $v_0\in V(I_i)$.
Since $v_0\in V(I)\setminus\{u,v\}$ and $x,y\notin V(I)$,
$v_0$ is neither split nor incident with the deleted edge $uv$.
Hence $d_{H_i}(v_0)=d_{H'}(v_0)=d_H(v_0)\ge5$.
\end{proof}

If $q_i>0$, then $I_i$ has a unique perfect matching
and $q_i\le q-2<q$.
By Claims~\ref{claim:boundary-connectivity}
and~\ref{claim:boundary-degree}, the induction hypothesis
gives $q_i\le2b_i+10h_i+10r_i-20$.
If $q_i=0$, then $b_i+r_i\ge5$ and $h_i\ge1$, so
$2b_i+10h_i+10r_i-20\ge0=q_i$.

\Needspace{10\baselineskip}
Put $s=|\{x,y\}\cap\partial V(H)|$.
By construction, $\sum_{i=1}^{a}b_i=b+s+(2-s)+2=b+4$
and $\sum_{i=1}^{a}r_i=r-(2-s)=r+s-2$.
The graph $H'$ has $|V(H)|+s$ vertices and $|E(H)|-1$ edges.
The faces of the graphs $H_i$ other than $f_{i,j}$
are precisely the faces of $H$ other than
$f_1,\ldots,f_h$ and the two faces bounded by $uvx,uvy$.
Applying Euler's formula to $H$ and each $H_i$ gives
$(|E(H)|-1)-(|V(H)|+s)+2a-\sum_{i=1}^{a}h_i
=|E(H)|-|V(H)|-h$.
Hence $\sum_{i=1}^{a}h_i=h-s+2a-1$, and therefore
$\sum_{i=1}^{a}(h_i+r_i-2)=h+r-3$.
Consequently,
$q=2+\sum_{i=1}^{a}q_i
\le2+2\sum_{i=1}^{a}b_i
+10\sum_{i=1}^{a}(h_i+r_i-2)
=2+2(b+4)+10(h+r-3)
=2b+10h+10r-20$.
This completes the induction.
\end{proof}

Taking $h=0$ gives the following consequence.

\begin{Pro}\label{pro:unique-deletion-bound}
Let $G$ be a simple planar 3-connected graph with $\delta(G)\ge5$.
If $W\subseteq V(G)$ and $G-W$ has a unique perfect matching,
then $|V(G)|\le11|W|-20$.
In particular, if $|W|\le k$, then $|V(G)|\le11k-20$.
\end{Pro}

\begin{proof}
Since $G$ is 3-connected, $|V(G)|\ge4$.
If $G-W$ is empty, then $W=V(G)$, so
$|V(G)|=|W|\le11|W|-20$.
Otherwise, fix a plane drawing of $G$ and take
$H=G$, $h=b=0$, and $Z=W$.
Then $\widehat H=G$ and $I=G-W$.
The graph $I$ is nonempty and has a unique perfect matching,
so $|V(I)|\ge2$.
Lemma~\ref{lem:boundary-unique} gives
$|V(G)|-|W|\le10|W|-20$,
and hence $|V(G)|\le11|W|-20$.
\end{proof}


\begin{samepage}
\subsection{\texorpdfstring{Proofs of Theorem~\ref{thm:degree-five} and
Corollary~\ref{cor:finite-family}}{Proofs of Theorem 2 and Corollary 3}}

\begin{proof}[Proof of Theorem~\ref{thm:degree-five}]
Put $p=\Phi(G)\ge1$. By Lemma~\ref{lem:halving}, there is
a set $W\subseteq V(G)$ with
$|W|\le2\lfloor\log_2p\rfloor$ such that $G-W$ has a
unique perfect matching. Proposition~\ref{pro:unique-deletion-bound}
therefore yields $|V(G)|\le11|W|-20
\le22\lfloor\log_2p\rfloor-20$.
Since $\lfloor\log_2p\rfloor$ is an integer, this implies
$\lfloor\log_2p\rfloor\ge\lceil(|V(G)|+20)/22\rceil$,
and hence $p\ge2^{\lceil(|V(G)|+20)/22\rceil}$.
\end{proof}

\begin{proof}[Proof of Corollary~\ref{cor:finite-family}]
Every such graph satisfies
$|V(G)|\le22\lfloor\log_2c\rfloor-20$ by
Theorem~\ref{thm:degree-five}. There are only finitely many
isomorphism classes of simple graphs of bounded order.
\end{proof}
\end{samepage}


\end{document}

%% file: figures/lemma11_selected_faces_GH.tex
\begin{tikzpicture}[
 x=.95cm,y=.95cm,
 every node/.style={font=\small},
 vertex/.style={circle,draw,fill=white,inner sep=0pt,minimum size=4.5pt},
 edge/.style={line width=.55pt},
 boundary/.style={draw=black!50,line width=1.1pt},
 added/.style={draw=black!65,dashed,line width=.55pt}]
\foreach \offset/\hatflag in {0/0,5.8/1} {
 \begin{scope}[shift={(\offset,0)}]
 \coordinate (u) at (-2,-1.7);
 \coordinate (v) at (2,-1.7);
 \coordinate (x) at (0,2.2);
 \coordinate (a) at (-.6667,-.0594);
 \coordinate (b) at (-.6667,-.6926);
 \coordinate (y) at (0,-.9661);
 \coordinate (c) at (.6667,-.6926);
 \coordinate (d) at (.6667,-.0594);
 \coordinate (w) at (0,-.045);
 \draw[boundary] (x)--(a)--(b)--(y)--(c)--(d)--cycle;
 \foreach \p in {x,a,b,y} {\draw[edge] (u)--(\p);}
 \foreach \p in {x,y,c,d} {\draw[edge] (v)--(\p);}
 \draw[line width=1.5pt] (u)--(v);
 \ifnum\hatflag=1\relax
   \foreach \p in {x,a,b,y,c,d} {\draw[added] (w)--(\p);}
   \node[vertex] at (w) {};
   \node[above right=1pt] at (w) {$w$};
 \fi
 \foreach \p in {x,a,b,y,c,d,u,v} {\node[vertex] at (\p) {};}
 \node[above] at (x) {$x$};
 \node[above left=1pt] at (a) {$a$};
 \node[left=1pt] at (b) {$b$};
 \node[below=1pt] at (y) {$y$};
 \node[right=1pt] at (c) {$c$};
 \node[above right=1pt] at (d) {$d$};
 \node[below left=1pt] at (u) {$u$};
 \node[below right=1pt] at (v) {$v$};
 \ifnum\hatflag=1\relax
   \node at (0,-2.5) {(b) $\widehat H$};
 \else
   \node at (0,-2.5) {(a) $H$};
 \fi
 \end{scope}
}
\begin{scope}[shift={(11,0)}]
\coordinate (u) at (-.8,0);
\coordinate (v) at (.8,0);
\draw[line width=1.5pt] (u)--(v);
\node[vertex] at (u) {};
\node[vertex] at (v) {};
\node[below] at (u) {$u$};
\node[below] at (v) {$v$};
\node at (0,-2.5) {(c) $I$};
\end{scope}
\end{tikzpicture}

%% file: figures/lemma_boundary_schematic_GH.tex
\begingroup
\def\schematicBeforePositions{w/0/2.8,w1/-3.05/1,u/-.85/1,v/.85/1,wd/3.05/1}
\def\schematicBeforeBoundary{w/w1,w/wd}
\def\schematicBeforeEdges{w/u,w/v}
\def\schematicAfterPositions{wu/-1.65/-.75,wv/1.65/-.75,w1/-3.05/-2.25,u/-.85/-2.25,v/.85/-2.25,wd/3.05/-2.25}
\def\schematicAfterBoundary{wu/w1,wv/wd}
\def\schematicAfterEdges{wu/u,wv/v}
\begin{tikzpicture}[
 x=1cm,y=.8cm,
 every node/.style={font=\small},
 vertex/.style={circle,draw,fill=white,inner sep=0pt,minimum size=5pt},
 split vertex/.style={vertex,line width=1pt},
 edge/.style={line width=.55pt},
 boundary/.style={draw=black!48,line width=1.2pt},
 matching/.style={line width=1.55pt}]

\begin{scope}
 \coordinate (x) at (0,2.6);
 \coordinate (y) at (0,-2.6);
 \coordinate (u) at (-.65,0);
 \coordinate (v) at (.65,0);
 \fill[black!9] (x)--(u)--(y)--(v)--cycle;
 \draw[edge] (x)--(u)--(y)--(v)--(x);
 \draw[matching] (u)--(v);
 \draw[boundary] (x)--(-2.15,1.4)--(-2.65,.45);
 \draw[boundary] (-2.65,-.45)--(-2.15,-1.4)--(y);
 \draw[boundary] (x)--(2.15,1.4)--(2.65,.45);
 \draw[boundary] (2.65,-.45)--(2.15,-1.4)--(y);
 \foreach \xx in {-2.65,2.65} {\node at (\xx,0) {$\vdots$};}
 \foreach \xx in {-2.15,2.15} {
  \foreach \yy in {-1.4,1.4} {\node[vertex] at (\xx,\yy) {};}
 }
 \foreach \yy in {-1,1} {
  \draw[edge] (u)--(-1.65,\yy);
  \draw[edge] (v)--(1.65,\yy);
  \node[vertex] at (-1.65,\yy) {};
  \node[vertex] at (1.65,\yy) {};
 }
 \node at (-1.65,0) {$\vdots$};
 \node at (1.65,0) {$\vdots$};
 \node[split vertex] at (x) {};
 \node[split vertex] at (y) {};
 \node[vertex] at (u) {};
 \node[vertex] at (v) {};
 \node[above=2pt] at (x) {$x$};
 \node[below=2pt] at (y) {$y$};
 \node[above right=2pt] at (u) {$u$};
 \node[above left=2pt] at (v) {$v$};
 \node at (-2.05,2.75) {$f_k$};
 \node[left=2pt] at (-2.45,1.2) {$C_k$};
 \node at (0,-3.45) {(a) $H$};
\end{scope}

\begin{scope}[shift={(7.8,0)}]
 \foreach \name/\xx/\yy in \schematicBeforePositions {
  \coordinate (\name) at (\xx,\yy);
 }
 \foreach \aa/\bb in \schematicBeforeBoundary {\draw[boundary] (\aa)--(\bb);}
 \foreach \aa/\bb in \schematicBeforeEdges {\draw[edge] (\aa)--(\bb);}
 \foreach \name/\xx/\yy in \schematicBeforePositions {\node[vertex] at (\name) {};}
 \node[split vertex] at (w) {};
 \node[above=2pt] at (w) {$w$};
 \node at (-1.65,2.7) {$f_k$};
 \node[below=3pt] at (w1) {$w_1$};
 \node[below=3pt] at (u) {$w_j=u$};
 \node[below=3pt] at (v) {$w_{j+1}=v$};
 \node[below=3pt] at (wd) {$w_d$};
 \node at (-1.95,1) {$\cdots$};
 \node at (1.95,1) {$\cdots$};
 \node at (0,-.05) {$\Downarrow$};
 \foreach \name/\xx/\yy in \schematicAfterPositions {
  \coordinate (\name) at (\xx,\yy);
 }
 \foreach \aa/\bb in \schematicAfterBoundary {\draw[boundary] (\aa)--(\bb);}
 \foreach \aa/\bb in \schematicAfterEdges {\draw[edge] (\aa)--(\bb);}
 \foreach \name/\xx/\yy in \schematicAfterPositions {\node[vertex] at (\name) {};}
 \node[split vertex] at (wu) {};
 \node[split vertex] at (wv) {};
 \node[above=2pt] at (wu) {$w_u$};
 \node[above=2pt] at (wv) {$w_v$};
 \node[below=3pt] at (w1) {$w_1$};
 \node[below=3pt] at (u) {$w_j=u$};
 \node[below=3pt] at (v) {$w_{j+1}=v$};
 \node[below=3pt] at (wd) {$w_d$};
 \node at (-1.95,-2.25) {$\cdots$};
 \node at (1.95,-2.25) {$\cdots$};
 \node at (0,-3.45) {(b) Splitting $w$};
\end{scope}

\begin{scope}[shift={(.55,-7.25)}]
 \coordinate (xu) at (0,2.6);
 \coordinate (yu) at (0,-2.6);
 \coordinate (u) at (.7,0);
 \draw[boundary] (xu)--(-2.15,1.4)--(-2.65,.45);
 \draw[boundary] (-2.65,-.45)--(-2.15,-1.4)--(yu);
 \draw[boundary] (xu)--(u)--(yu);
 \node at (-2.65,0) {$\vdots$};
 \foreach \yy in {-1.4,1.4} {\node[vertex] at (-2.15,\yy) {};}
 \foreach \yy in {-1,1} {
  \draw[edge] (u)--(-.85,\yy);
  \node[vertex] at (-.85,\yy) {};
 }
 \node at (-.85,0) {$\vdots$};
 \node at (-1.6,0) {$I_1$};
 \node[split vertex] at (xu) {};
 \node[split vertex] at (yu) {};
 \node[vertex] at (u) {};
 \node[above=2pt] at (xu) {$x_u$};
 \node[below=2pt] at (yu) {$y_u$};
 \node[right=2pt] at (u) {$u$};
 \node at (1.05,1.5) {$f_{1,1}$};
 \node at (-.55,-3.45) {(c) $H_1$};
\end{scope}

\begin{scope}[shift={(7.25,-7.25)}]
 \coordinate (xv) at (0,2.6);
 \coordinate (yv) at (0,-2.6);
 \coordinate (v) at (-.7,0);
 \draw[boundary] (xv)--(2.15,1.4)--(2.65,.45);
 \draw[boundary] (2.65,-.45)--(2.15,-1.4)--(yv);
 \draw[boundary] (xv)--(v)--(yv);
 \node at (2.65,0) {$\vdots$};
 \foreach \yy in {-1.4,1.4} {\node[vertex] at (2.15,\yy) {};}
 \foreach \yy in {-1,1} {
  \draw[edge] (v)--(.85,\yy);
  \node[vertex] at (.85,\yy) {};
 }
 \node at (.85,0) {$\vdots$};
 \node at (1.6,0) {$I_2$};
 \node[split vertex] at (xv) {};
 \node[split vertex] at (yv) {};
 \node[vertex] at (v) {};
 \node[above=2pt] at (xv) {$x_v$};
 \node[below=2pt] at (yv) {$y_v$};
 \node[left=2pt] at (v) {$v$};
 \node at (-1.05,1.5) {$f_{2,1}$};
 \node at (.55,-3.45) {(d) $H_2$};
\end{scope}
\node at (3.9,-11.6) {$H'=H_1\cup H_2$};
\end{tikzpicture}
\endgroup